\documentclass[11pt, a4paper]{amsart}

\usepackage[english]{babel}

\usepackage{microtype}

\usepackage[
        giveninits=true,
        citestyle=alphabetic,
        doi=false,
        url=false,
        bibstyle=alphabetic
]{biblatex}
\bibliography{higher_order_2.bib}

\usepackage{IEEEtrantools}

\usepackage{mathtools}

\usepackage{amssymb, amsmath, amsthm}

\newtheoremstyle{break}
  {}
  {}
  {\itshape}
  {}
  {\bfseries}
  {.}
  {\newline}
  {}

\theoremstyle{break}
\newtheorem{defn}{Definition}[section]
\newtheorem{cor}[defn]{Corollary}
\newtheorem{lem}[defn]{Lemma}

\newtheorem{thm}[defn]{Theorem}

\newcommand{\abs}[1]{\left| #1 \right|}
\newcommand{\CN}{\mathbb{C}}
\newcommand{\D}{\mathrm{d}}

\newcommand{\FT}{\operatorname{\mathcal{F}}}

\newcommand{\iu}{\mathrm{i}}
\newcommand{\iv}[1]{\frac{1}{ #1}}
\newcommand{\jb}[1]{\langle #1 \rangle}

\newcommand{\N}{\mathbb{N}}

\newcommand{\norm}[1]{\left\Vert #1 \right\Vert}

\newcommand{\R}{\mathbb{R}}
\newcommand{\set}[1]{\left\{ #1 \right\}}
\newcommand{\setp}[2]{\left\{ #1 \, \Big| \, #2 \right\}}
\DeclareMathOperator{\supp}{supp}
\DeclareMathOperator{\opS}{S}
\DeclareMathOperator{\opT}{T}
\newcommand{\Z}{\mathbb{Z}}

\title[Anisotropic higher order NLS and modulation spaces]
{Higher order NLS with anisotropic dispersion and modulation
spaces: a global existence and scattering result}

\subjclass[2010]{Primary 35A01, 35A02, 35Q55; Secondary 35B40.} 
\keywords{Global existence, Modulation spaces, Nonlinear dispersive equation,
Scattering}

\author{Chaichenets, Leonid}
\address{Leonid Chaichenets, Technical University of Dresden,
Institute of Analysis, 01069 Dresden, Germany}
\email{leonid.chaichenets@tu-dresden.de}

\author{Pattakos, Nikolaos}
\address{Nikolaos Pattakos, Department of Mathematics, Institute for
Analysis, Karlsruhe Institute of Technology, 76128 karlsruhe, Germany}
\email{nikolaos.pattakos@gmail.com}

\begin{document}
\begin{abstract}
In this paper we transfer a small data global existence and
scattering result by Wang and Hudzik to the more general case of
modulation spaces $M_{p, q}^s(\R^d)$ where $q = 1$ and $s \geq 0$ or
$q \in (1, \infty]$ and $s > \frac{d}{q'}$ and to the nonlinear
Schrödinger equation with higher order anisotropic dispersion.
\end{abstract}

\maketitle
\section{Introduction and main results}
We are interested in the following Cauchy problem for the higher-order
nonlinear Schrödinger equation (NLS) with anisotropic dispersion
\begin{equation}
\label{eqn:cauchy_nls_ho}
\left\{
\begin{IEEEeqnarraybox}[][c]{rCl}
\iu \partial_t u + \alpha \Delta u +
\iu \beta \frac{\partial^3 u}{\partial x_1^3} +
\gamma \frac{\partial^4 u}{\partial x_1^4} + f(u) & = & 0, \\
u(t = 0, \cdot) & = & u_0,
\end{IEEEeqnarraybox}
\right.
\end{equation}
where $u = u (t, x)$,
$(t, x) = (t, x_1, x') \in \R \times \R \times \R^{d - 1}$, $d \geq 2$,
$\alpha \in \R \setminus \set{0}$ and
$(\beta, \gamma) \in  \R^2 \setminus \set{(0, 0)}$.
We consider power-like and exponential non-linearities $f$. The former
non-linearity,
\begin{equation}
\label{eqn:power_nonlin}
f(u) = \pi^{m + 1}(u),
\end{equation}
is any product of in total $m + 1 \in \N$ copies of $u$ and
$\overline{u}$ of any sign, e.g. $f(u) = -\abs{u}^2u$. The latter
nonlinearity has the form
\begin{equation}
\label{eqn:exp_nonlin}
f(u) = \lambda\left(e^{\rho \abs{u}^2} - 1\right) u,
\end{equation}
where $\lambda \in \CN$ and $\rho > 0$.
Such PDEs arise in the context of high-speed
soliton transmission in long-haul optical communication systems and were
introduced by Karpman, see \cite{karpman1996}, \cite{diaz2008} and
\cite{karlsson1994}. The case where the coefficients $\alpha, \beta, \gamma$ are
time-dependent has been studied in \cite{carvajal2015} in one dimension
for the cubic nonlinearity, $f(u) = \abs{u}^2 u$ with initial data in
$L^2(\R)$-based Sobolev spaces. Before we state our results, we need to recall some
concepts and make some definitions.

The initial data $u_0$ in our case comes from a \emph{modulation space}. Modulation
spaces were introduced by Feichtinger in \cite{feichtinger1983} (see
also \cite{groechenig2001} or \cite{wang2007} for a gentle
introduction). Let $Q_0 \coloneqq \left(-\iv{2}, \iv2 \right]^d$ and
$Q_k \coloneqq k + Q_0$ for $k \in \Z^d$. Consider any smooth partition
of unity $(\sigma_k) \in C^\infty(\R^d)^{\Z^d}$ which is adapted to
$(Q_k)$, i.e. $\supp(\sigma_k) \subseteq B_{\sqrt{d}}(k)$ and
$\abs{\sigma_k(\xi)} \geq C$ for some $C > 0$, all $k \in \Z^d$ and
all $\xi \in Q_k$. Define the \emph{isometric decomposition operators}
\begin{equation}
\label{eqn:idc}
\Box_k \coloneqq \FT^{-1} \sigma_k \FT \qquad \forall k \in \Z^d,
\end{equation}
which we also call the \emph{box operators}. For $p, q \in [1, \infty]$
and $s \in \R$ we define the modulation space
\begin{eqnarray}
\label{eqn:modspace}
M_{p, q}^s(\R^d) & \coloneqq &
\setp{f \in S'(\R^d)}{\norm{f}_{M_{p, q}^s(\R^d)} < \infty},
\qquad \text{where} \\
\label{eqn:modnorm}
\norm{f}_{M_{p, q}^s(\R^d)} & \coloneqq &
\norm{\left(\jb{k}^s \norm{\Box_k f}_p \right)_k}_{l^q(\Z^d)},
\end{eqnarray}
$S'(\R^d)$ is the space of \emph{tempered distributions} and
$\jb{k} \coloneqq \sqrt{1 + \abs{k}^2}$ denotes the
\emph{Japanese bracket}. We note that a different partition of unity
leads to an equivalent norm. We shall sometimes shorten
$M_{p, q}^s(\R^d)$ to $M_{p, q}^s$ and $M_{p, q}^0$ to $M_{p, q}$.

The solution $u$ is sought in a \emph{Planchon-type space}. Spaces of
this type go back to \cite{planchon2000}. The following variant adapted
to modulation spaces has been introduced in \cite{wang2007}. For
$p, q, r \in [1, \infty]$ and $s \in \R$ we define
\begin{equation}
\label{eqn:planchon_space}
l_{\Box}^{s, q}(L^r(\R, L^p(\R^d))) \coloneqq
\setp{u \in S'(\R^{d + 1})}
{\norm{u}_{l_{\Box}^{s, q}(L^r(\R, L^p(\R^d)))} < \infty},
\end{equation}
where
\begin{equation}
\label{eqn:planchon_norm}
\norm{u}_{l_{\Box}^{s, q}(L^r(\R, L^p(\R^d)))} \coloneqq
\norm{\left(\jb{k}^s \norm{(\Box_k u)(t, x)}_{L_t^r L_x^p}
\right)_{k \in \Z^d}}_{l^q}
\end{equation}
(the box operator acts in the space variable $x$ only). As already done
above we shall indicate in which variable which norm is taken by writing
e.g. $L^r_t L^p_x$ for the mixed-norm space $L^r(\R, L^p(\R^d))$. Also,
we shall sometimes shorten $l^{s, q}_\Box(L^r(\R, L^p(\R^d)))$ to
$l^{s, q}_\Box(L^r L^p)$.

Finally, we denote by $\lceil \cdot \rceil$ the ceiling and by
$\lfloor \cdot \rfloor$ the floor functions.

We are now in the position to state our results.
\begin{thm}
\label{thm:power}
Let $d \in \N$ with $d \geq 2$, $f$ be a power-like nonlinearity from
Equation
\eqref{eqn:power_nonlin} and assume
\begin{equation}
\label{eqn:minimal_power}
m \geq m_0 \coloneqq
\left\lceil \frac{4}{d - \iv{c_\gamma}} \right\rceil,
\qquad \text{where} \qquad
c_\gamma =
\begin{cases}
2, & \text{for } \gamma \neq 0, \\
3, & \text{for } \gamma = 0.
\end{cases}.
\end{equation}
Moreover, let
$\iv{r} \in I_{m ,d}$ and $\iv{p} \in J_{r, d}$,
where
\begin{eqnarray}
\label{eqn:time_range}
I_{m, d} & \coloneqq &
\left[\iv{2(m + 1)}, \iv{m_0 + 1} \right], \\
\label{eqn:space_range}
J_{r, d} & \coloneqq &
\left[
\iv{2} - \frac{2}{r \left(d - \iv{c_\gamma}\right)} -
\iv{2 (l + 1)} \left(l - \frac{4}{d - \iv{c_\gamma}}
\right), \right. \\
& & \left.
\iv{2} - \frac{2}{r \left(d - \iv{c_\gamma}\right)}
\right] \qquad \text{and} \\
\label{eqn:effective_nonlinearity}
l & \coloneqq & \min \set{\lfloor r \rfloor - 1, m}.
\end{eqnarray}
If $q = 1$ let $s \geq 0$ and if $q > 1$ let $s > \frac{d}{q'}$.

Then, there exists a $\delta > 0$ such that for any
$u_0 \in M_{2, q}^s$  with $\norm{u_0}_{M_{2, q}^s} \leq \delta$
the Cauchy problem \eqref{eqn:cauchy_nls_ho} has a unique global (mild)
solution in
\begin{equation}
\label{eqn:solspace}
X \coloneqq
l_\Box^{s, q}(L^\infty(\R, L^2(\R^d))) \cap
l_\Box^{s, q}(L^r(\R, L^p(\R^d)))
\subseteq C(\R, M_{2, q}^s(\R^d)).
\end{equation}
\end{thm}

\begin{thm}
\label{thm:exponential}
Let $d \geq 2$, $f$ be an exponential non-linearity from Equation
\eqref{eqn:exp_nonlin}. Moreover, let $\iv{r} \in I_{3, d}$ and
$\iv{p} \in J_{r, d}$ (see Equation \eqref{eqn:minimal_power},
\eqref{eqn:time_range} and \eqref{eqn:space_range}). If $q = 1$ let
$s \geq p$ and if $q > 1$ let $s > \frac{d}{q'}$.

Then, there exists a $\delta > 0$ such that for any $u_0 \in M_{2, q}^s$
with $\norm{u_0}_{M_{2, q}^s} \leq \delta$ the Cauchy problem
\eqref{eqn:cauchy_nls_ho} has a unique global (mild) solution in $X$
(given by Equation \eqref{eqn:solspace}).
\end{thm}

\begin{cor}
\label{cor:scattering}
If, in Theorem \ref{thm:power}, one additionally has $q \leq m + 1$ then the
scattering operator $S$ carries a whole neighbourhood of $0$ in $M_{2, q}^s(\R^d)$ into $M_{2,q}^{s}(\R^d)$. The same is true for Theorem \ref{thm:exponential}, if $q \leq 3$.
\end{cor}

\section{Preliminaries}
We start by the following observation.
\begin{lem}
\label{lem:minkowski}
Let $l, p, q \in [1, \infty]$ with $q \leq l$ and $s \in \R$. Then
\begin{equation*}
l^{s, q}_\Box(L^l(\R, L^p(\R^d)) \hookrightarrow
L^l(\R, M_{p, q}^s(\R^d)).
\end{equation*}
\end{lem}
\begin{proof}
This is just an application of Minkowski’s integral inequality to $L^l$
and $l^{s, q}$.
\end{proof}
Let us denote by $W(t)$ the \emph{free Schrödinger propagator with
higher-order anisotropic dispersion} at time $t \in \R$, i.e.
\begin{equation}
W(t) =
\FT^{-1} e^{\iu \left(
\alpha \abs{\xi}^2 + \beta \xi_1^3 + \gamma \xi_1^4
\right) t} \FT.
\end{equation}
We cite the \emph{Strichartz estimates} for this propagator from
\cite[Theorem 1.2]{bouchel2008}. Let us remark that the dispersive
estimate (from which the Strichartz estimates follow), was obtained in
\cite{ben2000} for the isotropic case. Strichartz estimates hold
for the so-called (dually) \emph{admissible pairs}. We call the pair
$(p, r) \in [2, \infty] \times [2, \infty]$ admissible, if
\begin{equation}
\label{eqn:admissible}
\frac{2}{r} + \left(d - \iv{c_\gamma} \right) \iv{p} =
\left(d - \iv{c_\gamma} \right) \iv{2}, 
\end{equation}

\begin{thm}[Strichartz estimates]
\label{thm:strichartz}
Let $(p, r)$ and $(\tilde{p}', \tilde{r}')$ be admissible. For any
$u_0 \in L^2(\R^d)$ one has the \emph{homogeneous Strichartz estimate}
\begin{equation}
\label{eqn:homogeneous_strichartz}
\norm{W(t) u_0}_{L_t^r L^p_x} \lesssim \norm{u_0}_{L^2}
\end{equation}
and for any $F \in L^{\tilde{r}}_t L^{\tilde{p}}_x$ one has the
\emph{inhomogeneous Strichartz estimate}
\begin{equation}
\label{eqn:inhomogeneous_strichartz}
\norm{\int_0^t W(t - \tau) F(\tau, \cdot) \D{\tau}}_{L_t^r L_x^p}
\lesssim
\norm{F}_{L_t^{\tilde{r}} L_x^{\tilde{p}}}
\end{equation}
with implicit constants independent of $u_0$ and $F$.
\end{thm}

Strichartz estimates on Lebesque spaces immediately translate to the
setting of modulation and Planchon-type spaces.
\begin{lem}
Let $(p, r)$ and  $(\tilde{p}', \tilde{r}')$ be admissible. Moreover,
let $q \in [1, \infty]$ and $s \in \R$. For any
$u_0 \in M_{p, q}^s$ one has
\begin{equation}
\label{eqn:hom_strichartz_planchon}
\norm{W(t) u_0}_{l^{s, q}_\Box(L^r_t L^p_x)} \lesssim
\norm{u_0}_{M_{2, q}^s}
\end{equation}
and for any
$F \in l^{s, q}_\Box(L^{\tilde{r}}(\R, L^{\tilde{p}}(\R^d)))$ one has
\begin{equation}
\label{eqn:inhom_strichartz_planchon}
\norm{\int_0^t W(t - \tau) F(\tau, \cdot) \D{\tau}}_{
l^{s, q}_\Box(L^r_t L^p_x)}
\lesssim
\norm{F}_{
l^{s, q}_\Box(L^{\tilde{r}} L^{\tilde{p}})}
\end{equation}
with implicit constants independent of $u_0$ and $F$.
\end{lem}
\begin{proof}
The proof follows from Theorem \ref{thm:strichartz} and the
fact that $W(t)$ and $\Box_k$ commute.
\end{proof}

\begin{lem}[Hölder-like inequalities]
\label{lem:hoelderlike}
Let $d, n \in \N$ and
$\tilde{p}, \tilde{p}_1, \ldots. \tilde{p}_n, q, \tilde{r},
\tilde{r}_1, \ldots, \tilde{r}_n \in [1, \infty]$ be such
that
\begin{equation*}
\iv{\tilde{p}} = \iv{\tilde{p}_1} + \cdots + \iv{\tilde{p}_n}, \qquad
\iv{\tilde{r}} = \iv{\tilde{r}_1} + \cdots + \iv{\tilde{r}_n}.
\end{equation*}
For $q > 1$ let $s > d \left(1 - \iv{q} \right)$ and for $q = 1$ let
$s \geq 0$.
Then, for any
\begin{equation*}
(f_1, \ldots, f_n) \in
l_\Box^{s, q}(L^{\tilde{r}_1}(\R, L^{\tilde{p}_1}(\R^d))
\times \cdots \times
l_\Box^{s, q}(L^{\tilde{r}_n}(\R, L^{\tilde{p}_n}(\R^d))
\end{equation*}
one has
\begin{equation}
\label{eqn:hoelderlike}
\norm{\prod_{j = 1}^n f_j}_{
l_\Box^{s, q}(L^{\tilde{r}} L^{\tilde{p}})}
\lesssim
\prod_{j = 1}^n\norm{f_j}_{
l_\Box^{s, q}\left(L^{\tilde{r}_j} L^{\tilde{p}_j}\right)}
\end{equation}
with an implicit constant independent of $(f_1. \ldots, f_n)$.

Similarly, for
\begin{equation*}
(g_1, \ldots, g_n) \in
M^s_{\tilde{p}_1, q}(\R^d)
\times \cdots \times
M^s_{\tilde{p}_n, q}(\R^d)
\end{equation*}
one has
\begin{equation}
\label{eqn:hoelder_mod}
\norm{\prod_{j = 1}^n g_j}_{
M^s_{\tilde{p}, q}}
\lesssim
\prod_{j = 1}^n \norm{g_j}_{
M^s_{\tilde{p}_j, q}}
\end{equation}
with an implicit constant independent of $(g_{1},\ldots,g_{n})$.
\end{lem}
\begin{proof}
The special case of \eqref{eqn:hoelder_mod} for $n = 2$ is proven in
\cite[Theorem 4.3]{chaichenets2018}. That proof is easily
transferred to the case of Planchon-type spaces applying Hölder’s
inequality for the time variable, also.
\end{proof}

\begin{lem}
\label{lem:bernstein}
Let $p_1, p_2, q, r \in [1, \infty]$ with $p_1 \leq p_2$ and $s \in \R$.
Then
\begin{equation}
\label{eqn:planchon_embedding}
l^{s, q}_\Box(L^r(\R, L^{p_1}(\R^d))) \hookrightarrow
l^{s, q}_\Box(L^r(\R, L^{p_2}(\R^d))).
\end{equation}
\end{lem}
\begin{proof}
The conclusion follows immediately from the definition of the norm and
Bernstein’s multiplier estimate (see e.g.
\cite[Corollary A.53]{chaichenets2018}).
\end{proof}

\begin{cor}
Let $d \in \N$, $m \in \N_0$ and $l \in \set{0, \ldots, m}$. Moreover,
let $\tilde{p}, q, \tilde{r} \in [1, \infty]$. For $q > 1$ let
$s > d \left(1 - \iv{q}\right)$ and for $q = 1$ let $s \geq 0$. Then
\begin{eqnarray}
\nonumber
\norm{\pi^{m + 1}(u) - \pi^{m + 1}(v)
}_{l^{s, q}_\Box (L^{\tilde{r}} L^{\tilde{p}})}
\label{eqn:power_lipschitz}
& \lesssim &
\norm{u - v}_{
l^{s, q}_\Box(L^{(l + 1) \tilde{r}} L^{(l + 1) \tilde{p}})} \\
& & \cdot
\left(
\norm{u}_{l^{s, q}_\Box \left(L^{(l + 1) \tilde{r}} L^{(l + 1) \tilde{p}}\right)}^l
\norm{u}_{l^{s, q}_\Box (L^\infty L^2)}^{m - l} \right. \\
\nonumber
& & +
\left. \norm{v}_{l^{s, q}_\Box (L^{(l + 1) \tilde{r}} L^{(l + 1) \tilde{p}})}^l
\norm{v}_{l^{s, q}_\Box (L^\infty L^2)}^{m - l}
\right)
\end{eqnarray}
for any $u, v \in S'(\R^{d + 1})$, where the implicit constant is
independent of $(u, v)$.
\end{cor}
\begin{proof}
The base case $m = 0$ is trivial. For the induction step from
$m - 1$ to $m$ observe, that
\begin{eqnarray}
\nonumber
& & \norm{\pi^{m + 1}(u) - \pi^{m + 1}(v)
}_{l^{s, q}_\Box(L^{\tilde{r}} L^{\tilde{p}})} \\
\label{eqn:two_summands}
& \leq &
\norm{\pi^m(u)(u - v)
}_{l^{s, q}_\Box(L^{\tilde{r}} L^{\tilde{p}})} +
\norm{(\pi^m(u) - \pi^m(v))v
}_{l^{s, q}_\Box(L^{\tilde{r}} L^{\tilde{p}})}.
\end{eqnarray}
For the first summand one has, by the Hölder-like inequality
\eqref{eqn:hoelderlike},
\begin{eqnarray*}
& & \norm{\pi^m(u)(u - v)
}_{l^{s, q}_\Box(L^{\tilde{r}} L^{\tilde{p}})} \\
& \lesssim &
\norm{u - v}_{l^{s, q}_\Box(L^{(l + 1)\tilde{r}} L^{(l + 1) \tilde{p}})}
\norm{\pi^m(u)}_{l^{s, q}_\Box\left(L^{\frac{l + 1}{l}\tilde{r}}
L^{\frac{l + 1}{l} \tilde{p}}\right)}.
\end{eqnarray*}
For $l = 0$ the second factor above is estimated via
\eqref{eqn:hoelderlike} and \eqref{eqn:planchon_embedding} against
\begin{equation*}
\norm{\pi^m(u)}_{l^{s, q}_\Box\left(L^{\frac{l + 1}{l}\tilde{r}}
L^{\frac{l + 1}{l} \tilde{p}}\right)} \lesssim
\norm{u}_{l^{s, q}_\Box(L^\infty L^\infty)}^m \lesssim
\norm{u}_{l^{s, q}_\Box(L^\infty L^2)}^m.
\end{equation*}
For $l \geq 1$ one uses the induction hypothesis instead and arrives at
\begin{equation*}
\norm{\pi^m(u)}_{l^{s, q}_\Box\left(L^{\frac{l + 1}{l}\tilde{r}}
L^{\frac{l + 1}{l} \tilde{p}}\right)} \lesssim
\norm{u}_{l^{s, q}_\Box(L^{(l + 1) \tilde{r}} L^{(l + 1) \tilde{p}})}^l
\norm{u}_{l^{s, q}_\Box(L^\infty L^2)}^{m - l}.
\end{equation*}
The second summand in \eqref{eqn:two_summands} is treated via the same
methods. This concludes the proof.
\end{proof}

\section{Proofs of the results}
In this section we present the proofs of Theorem \ref{thm:power}, of
Theorem \ref{thm:exponential} and of Corollary \ref{cor:scattering}.

\begin{proof}[Proof of Theorem \ref{thm:power}]
For a $\delta > 0$, which will be fixed later, consider
\begin{equation}
\label{eqn:metric_space}
M(\delta) \coloneqq \setp{f \in X}{\norm{f}_X \leq \delta},
\end{equation}
where $X$ is given by \eqref{eqn:solspace} (the embedding claimed
there immediately follows from Lemma \ref{lem:minkowski} with
$l = \infty$ and $p = 2$). By Banach’s fixed-point theorem, it
suffices to show that the operator
\begin{equation}
\label{eqn:contraction}
u \mapsto \opT u \coloneqq
W(t) u_0 + \int_0^t W(t - \tau) \pi^{m + 1}(u(\tau)) \D{\tau}
\end{equation}
is a contractive self-mapping of $M(\delta)$ for some $\delta > 0$. We
only show the contractiveness of $\opT$, as the proof of the
self-mapping property follows along the same lines. Fix any
$u, v \in M(\delta)$.

By the definition of the norm $\norm{\cdot}_X$ one has
\begin{align*}
\norm{\opT(u) - \opT(v)}_X & =
\norm{\opT(u) - \opT(v)}_{l^{s, q}_\Box(L^\infty L^2)} +
\norm{\opT(u) - \opT(v)}_{l^{s, q}_\Box(L^r L^p)}.
\end{align*}
To apply Strichartz estimates we require spaces with admissible indices
in both summands. This is already the case for the first summand.
For the second summand, observe that Assumption \eqref{eqn:space_range}
implies
\begin{equation*}
p \geq p_\text{a} \coloneqq
\left(\iv{2} - \frac{2}{r \left(d - \iv{c_\gamma}\right)}\right)^{-1}
\end{equation*}
and hence, by Lemma \ref{lem:bernstein},
\begin{equation*}
\norm{\opT(u) - \opT(v)}_{
l^{s, q}_\Box(L^r L^p)} \lesssim
\norm{\opT(u) - \opT(v)}_{
l^{s, q}_\Box(L^r L^{p_{\text{a}})}}
\end{equation*}
follows. Due to $\iv{r} \in I_{m, d} \subseteq \left[0, \iv{2}\right]$,
the pair $(p_\text{a}, r)$ is indeed admissible.

By the inhomogeneous Strichartz estimate
\eqref{eqn:inhom_strichartz_planchon} we therefore have
\begin{eqnarray*}
\norm{\opT u - \opT v}_X & = &
\norm{\int_0^t W(t - \tau)
\left(\pi^{m + 1}(u(\tau)) - \pi^{m + 1}(v(\tau))\right) \D{\tau}}_X \\
& \lesssim &
\norm{\pi^{m + 1}(u) - \pi^{m + 1}(v)
}_{l^{s, q}_\Box(L^{\tilde{r}} L^{\tilde{p}})}
\end{eqnarray*}
for any pair $(\tilde{p}, \tilde{r}) \in [1, 2]$ such that
$\left(\tilde{p}', \tilde{r}'\right)$ is admissible, which we will fix
in the following. Observe, that
\begin{equation*}
I_{m, d} = \bigcup_{k = m_0}^m \left[\iv{2(k + 1)}, \iv{k + 1} \right]
\end{equation*}
and that the effective non-linearity $l$ satisfies
\begin{equation}
\label{eqn:maximal_effective_nonlinearity}
l = \max
\setp{k \in \set{m_0, \ldots, m}}
{\iv{r} \in \left[ \iv{2(k + 1)}, \iv{k + 1} \right]}.
\end{equation}
Hence, $\tilde{r} \coloneqq \frac{r}{l + 1} \in [1, 2]$ and thus
$(\tilde{p}', \tilde{r}')$, where
\begin{equation}
\label{eqn:dually_admissible}
\tilde{p} \coloneqq
\left(\iv{2} +
\frac{2 \left(1 - \iv{\tilde{r}} \right)}
{d - \iv{c_\gamma}}\right)^{-1},
\end{equation}
indeed form an admissible pair.

We now apply Estimate \eqref{eqn:power_lipschitz} to arrive at
\begin{eqnarray*}
& & \norm{\pi^{m + 1}(u) - \pi^{m + 1}(v)
}_{l^{s, q}_\Box(L^{\tilde{r}} L^{\tilde{p}})} \\
& \lesssim &
\norm{u - v}_{
l^{s, q}_\Box(L^r L^{(l + 1) \tilde{p}})} \\
& & \cdot
\left(
\norm{u}_{l^{s, q}_\Box \left(L^r L^{(l + 1) \tilde{p}}\right)}^l
\norm{u}_{l^{s, q}_\Box (L^\infty L^2)}^{m - l} +
\norm{v}_{l^{s, q}_\Box (L^r L^{(l + 1) \tilde{p}})}^l
\norm{v}_{l^{s, q}_\Box (L^\infty L^2)}^{m - l}
\right).
\end{eqnarray*}
A short calculation yields
\begin{equation*}
(l + 1) \tilde{p} =
\left(
\iv{2} - \frac{2}{r\left(d - \iv{c_\gamma}\right)} -
\iv{2(l + 1)} \left(l - \frac{4}{d - \iv{c_\gamma}} \right)
\right)^{-1},
\end{equation*}
and thus, due to Assumption \eqref{eqn:space_range},
$(l + 1) \tilde{p} \geq p$.
Therefore, invoking Embedding \eqref{eqn:planchon_embedding} once again,
one finally obtains
\begin{eqnarray*}
& &
\norm{\opT(u) - \opT(v)}_X \\
& \lesssim &
\norm{u - v}_{l^{s, q}_\Box(L^r L^p)} \\
& & \cdot
\left(\norm{u}_{l^{s, q}_\Box(L^r L^p)}^l
\norm{u}_{l^{s, q}_\Box(L^\infty L^2)}^{m - l} +
\norm{v}_{l^{s, q}_\Box(L^r L^p)}^l
\norm{v}_{l^{s, q}_\Box(L^\infty L^2)}^{m - l}\right) \\
& \lesssim &
\norm{u - v}_X \delta^m.
\end{eqnarray*}
Choosing $\delta$ small enough finishes the proof.
\end{proof}

\begin{proof}[Proof of Theorem \ref{thm:exponential}]
By the definition of the exponential nonlinearity \eqref{eqn:exp_nonlin} we have
\begin{equation}
f(u) =
\lambda \sum_{m = 1}^\infty \frac{\rho^m}{m!} \abs{u}^{2m} u.
\end{equation}
We have to show that the operator
\begin{equation}
u \mapsto \opT u \coloneqq W(t) u_0 + \int_0^t W(t - \tau) f(u(\tau)) \D{\tau}
\end{equation}
is a contractive self-mapping of the complete metric space $M(\delta)$ as defined in
\eqref{eqn:metric_space}. We only briefly sketch the argument for the self-mapping. We have
\begin{eqnarray*}
\norm{\opT u}_X & \lesssim &
\norm{u_0}_{M_{2, q}^s} +
\sum_{m = 1}^\infty \frac{\rho^m}{m!}
\norm{\int_0^t W(t - \tau) \pi^{2m + 1}(u) \D{\tau}}_X \\
& \lesssim &
\norm{u_0}_{M_{2, q}^s} +
\sum_{m = 1}^\infty \frac{\rho^m}{m!} \norm{u}_X^{2m + 1} \leq
\norm{u_0}_{M_{2, q}^s} + \left(\delta e^{\rho \delta^2} - 1 \right) \\
& \overset{!}{\leq} & \delta,
\end{eqnarray*}
where we set $\pi^{2m + 1}(u) \coloneqq \abs{u}^{2m} u$ and proceeded as in the proof of
Theorem \ref{thm:power}. This is justified because the assumptions on $r$ and $p$ are
exactly such that Theorem \ref{thm:power} can be applied to \emph{any}
$\pi^{2m + 1}(u)$ with $m \geq 1$  ($I_{m, d} \subseteq I_{n, d}$ for $m \leq n$). Thus,
if $\norm{u_0}_{M_{2, q}^s} \lesssim \frac{\delta}{2}$ and $\delta > 0$ is sufficiently
small, the operator $\opT$ is a self-mapping of $M(\delta)$.
\end{proof}

\begin{proof}[Proof of Corollary \ref{cor:scattering}]
Given initial data $u_0^{-}$ at $-\infty$ of sufficiently small
$M_{2, q}^s(\R^d)$-norm we show that the operator $\opS_-$ given by
\begin{align}
(\opS_- u) (t) & \coloneqq
W(t) u_0^{-} + \int_{-\infty}^t W(t - \tau) f(u(\tau)) \D{\tau}, &
t & \in \R,
\end{align}
is a contractive self-mapping of $M$ given in \eqref{eqn:metric_space}.
As the only difference between $S_-$ and $\opT$ is the lower limit of the
integral being $-\infty$ instead of $0$, the argument is very similar to the
proof of Theorem \ref{thm:power} and we omit the details. We denote the unique
fixed-point of $S_-$ by $u$. Following the proof of Theorem \ref{thm:power} we
notice, that
$u \in l^{s, q}_\Box(L^{\rho} L^{\sigma})$ for any
$\rho, \sigma \in [2, \infty]$ satisfying
\begin{equation}
\label{eqn:subadmissible}
\frac{2}{\rho} + \left(d - \iv{c_\gamma}\right) \iv{\sigma} \leq
\left(d - \iv{c_\gamma}\right) \iv{2}.
\end{equation}
This is because we can replace the $\sigma$ by $\sigma_\text{a}$ such that
$(\sigma_\text{a}, \rho)$ is admissible, i.e.
\begin{equation*}
\sigma \geq \sigma_\text{a} \coloneqq
\left(\iv{2} - \frac{2}{\rho \left(d - \iv{c_\gamma}\right)} \right)^{-1},
\end{equation*}
and hence
\begin{eqnarray*}
\norm{u}_{l^{s, q}_\Box (L^{\rho} L^{\sigma})} & \lesssim &
\norm{S_- u}_{l^{s, q}_\Box (L^{\rho} L^{\sigma_\text{a}})} \lesssim
\norm{u_0^-}_{M_{2, q}^s} +
\norm{\pi^{m + 1}(u)}_{l^{s, q}_\Box(L^{\tilde{r}} L^{\tilde{p}})} \\
& \lesssim &
\norm{u}_X^{m + 1} < \infty,
\end{eqnarray*}
where above we used Lemma \ref{lem:bernstein}, Strichartz estimates and the
Hölder-like inequality for Planchon-type spaces. Notice that, as
$t \to -\infty$, one has
\begin{equation}
\norm{u(t) - W(t)u_0^-}_{M_{2, q}^s} \to 0.
\end{equation}
This is because by Lemma \ref{lem:minkowski} and the assumption
$q \leq m + 1$ one has
\begin{eqnarray*}
& & \int_{-\infty}^t \norm{\pi^{m + 1}(u)(\cdot, \tau)}_{M_{2, q}^s} \D{\tau}
\leq \int_{-\infty}^\infty \norm{u(\cdot, \tau)}_{M_{2(m + 1), q}^s}^{m + 1} \D{\tau} \\
& = &
\norm{u}_{L^{m + 1}\left(M_{2(m + 1), q}^s\right)}^{m + 1} \leq
\norm{u}_{l^{s, q}_\Box(L^{m + 1}_t L^{2(m + 1)}_x)}^{m + 1}
\end{eqnarray*}
and, as $\rho = m + 1$ and $\sigma = 2(m + 1)$ satisfy the condition
\eqref{eqn:subadmissible} by the prerequisite $m \geq m_0$, the norm above
is finite. Finally, we define
\begin{equation}
u_0^+ \coloneqq u_0^- + \int_{-\infty}^\infty W(-\tau) \pi^{m + 1}(u(\tau)) \D{\tau}
\in M_{2, q}^s(\R^d)
\end{equation}
and notice that $\norm{u(t) - W(t)u_0^+}_{M_{2, q}^s} \to 0$ as $t \to +\infty$,
because
\begin{eqnarray*}
\norm{u(t) - W(t)u_0^+}_{M_{2, q}^s} & = &
\norm{W(-t) u(t) - u_0^+}_{M_{2, q}^s} \\
& = &
\norm{\int_t^{\infty} W(-\tau) \pi^{m + 1}(u(\tau)) \D{\tau}}_{M_{2, q}^s}
\end{eqnarray*}
and the same argument as for the convergence to $u_0^-$ applies. Hence, the scattering operator $u_0^- \overset{S}{\mapsto} u_0^+$ indeed carries a whole neighborhood of $0$ in $M_{2,q}^s$ into $M_{2,q}^s$.
\end{proof}

\section*{Acknowledgment}
The authors thank Professor Baoxiang Wang from the School of
Mathematical Sciences of the Peking University for fruitful discussions.

\printbibliography
\end{document}